\documentclass{birkmult}
 \newtheorem{thm}{Theorem}[section]
 \newtheorem{conjecture}[thm]{Conjecture}
 
 \newtheorem{lem}[thm]{Lemma}
 
 \theoremstyle{definition}
 
 \theoremstyle{remark}

 \numberwithin{equation}{section}

\begin{document}
%-------------------------------------------------------------------------
% editorial commands: to be inserted by the editorial office
%
%\firstpage{1}
%\volume{228}
%\Copyrightyear{2004}
%\DOI{003-0001}
%
%
%\seriesextra{Just an add-on}
%\seriesextraline{This is the Concrete Title of this Book\br H.E. R and S.T.C. W, Eds.}
%
% for journals:
%
%\firstpage{1}
%\issuenumber{1}
%\Volumeandyear{1 (2004)}
%\Copyrightyear{2004}
%\DOI{003-xxxx-y}
%\Signet
%\commby{inhouse}
\submitted{November 15, 2004}
%\received{March 16, 2000}
%\revised{June 1, 2000}
%\accepted{July 22, 2000}
%
%
%
%---------------------------------------------------------------------------
%Insert here the title, affiliations and abstract:
%
\title%[An Example for birkmult]%titre raccourci pour les entetes 
{Odd pairs of cliques}
%----------Author 1
\author[Burlet]{Michel Burlet}
%\address{%
%Laboratoire Leibniz\\
%46 avenue F\'elix Viallet\\
%38031 Grenoble cedex\\
%France
%}
\email{michel.burlet@imag.fr}

%\thanks{This work was completed with the support of our
%\TeX-pert.}

%----------Author 2
\author[Maffray]{Fr\'ed\'eric Maffray}
%\address{%
%Laboratoire Leibniz\\
%46 avenue F\'elix Viallet\\
%38031 Grenoble cedex\\
%France
%}
%
\email{frederic.maffray@imag.fr}

%----------Author 3
\author[Trotignon]{Nicolas Trotignon}
%\address{%
%\mbox{}\vspace{3ex}
%Laboratoire Leibniz\\
%46 avenue F\'elix Viallet\\
%38031 Grenoble cedex\\
%France}
%
\email{nicolas.trotignon@imag.fr}

\address{%
\mbox{}\vspace{1ex}\\
Laboratoire Leibniz\\
46 avenue F\'elix Viallet\\
38031 Grenoble cedex\\
France}

%----------classification, keywords, date
%\subjclass{Primary 99Z99; Secondary 00A00}

\keywords{Perfect graph, graph, even pair}

\date{November 15, 2004}
%----------additions
%\dedicatory{To my boss}
%%% ----------------------------------------------------------------------

\begin{abstract}
A graph is Berge if it has no induced odd cycle on at least 5~vertices
and no complement of induced odd cycle on at least 5~vertices. A graph
is perfect  if the chromatic  number equals the maximum  clique number
for every induced subgraph.  Chudnovsky, Robertson, Seymour and Thomas
proved that every Berge graph  either falls into some classical family
of perfect graphs,  or has a structural fault that  cannot occur in a
minimal imperfect  graph. A  corollary of this  is the  strong perfect
graph theorem conjectured by Berge:  every Berge graph is perfect.  An
even pair of vertices in a graph is a pair of vertices such that every
induced  path between  them has  even  length. Meyniel  proved that  a
minimal imperfect graph cannot contain an even pair. So even pairs may
be considered  as a  structural fault.  Chudnovsky  et al. do  not use
them, and  it is known  that some classes  of Berge graph have no
even pairs.

The aim  of this work  is to investigate an  ``even-pair-like'' notion
that could be a structural fault  present in every Berge graph. An odd
pair of  cliques is a pair  of cliques $\{K_1, K_2\}$  such that every
induced path from $K_1$ to $K_2$  with no interior vertex in $K_1 \cup
K_2$ has odd length.  We conjecture  that for every Berge graph $G$ on
at least  two vertices,  either one of  $G, \overline{G}$ has  an even
pair, or  one of  $G, \overline{G}$  has an odd  pair of  cliques.  We
conjecture that a  minimal imperfect graph has no  odd pair of maximal
cliques.  We prove  these conjectures in some special  cases.  We show
that adding  all edges between any 2~vertices of the cliques  of an odd
pair of cliques is an operation that preserves perfectness.
\end{abstract}

%%% ----------------------------------------------------------------------
\maketitle
%%% ----------------------------------------------------------------------
%\tableofcontents

%debut de l'article proprement dit

\section{Introduction}

In  this paper  graphs  are  simple, non-oriented,  with  no loop  and
finite.  Several definitions that can  be found in most handbooks (for
instance~\cite{diestel:graph})  will not  be  given.  A  graph $G$  is
\emph{perfect}  if  every  induced  subgraph  $G'$  of  $G$  satisfies
$\chi(G')=\omega(G')$,  where $\chi(G')$  is the  chromatic  number of
$G'$   and  $\omega(G')$  is   the  maximum   clique  size   in  $G'$.
Berge~\cite{berge:60,berge:61}    introduced   perfect    graphs   and
conjectured  that the  complement  of  a perfect  graph  is a  perfect
graph. This conjecture was proved by Lov\'asz:

\begin{thm}[Lov\'asz, \cite{lovasz:nh,lovasz:pg}]
  \label{th:lovasz}
  The complement of every perfect graph is a perfect graph.
\end{thm}

Berge also conjectured a  stronger statement: \emph{a graph is perfect
if and only if it does not  contain as an induced subgraph an odd hole
or an  odd antihole}  (the Strong Perfect  Graph Conjecture),  where a
\emph{hole} is  a chordless cycle with  at least four  vertices and an
\emph{antihole} is the complement of  a hole.  We follow the tradition
of calling \emph{Berge graph} any  graph that contains no odd hole and
no odd antihole.  The Strong Perfect Graph Conjecture was the objet of
much research (see  the book~\cite{livre:perfectgraphs}), until it was
finally    proved    by    Chudnovsky,    Robertson,    Seymour    and
Thomas:

\begin{thm}[Chudnovsky,    Robertson,    Seymour    and
    Thomas~\cite{chudvovsky.r.s.t:spgt}]
  \label{th:fort}
  Every Berge graph is perfect.
\end{thm}

In       fact      Chudnovsky,       Robertson,       Seymour      and
Thomas~\cite{chudvovsky.r.s.t:spgt}    proved    a   stronger    fact,
conjectured        by         Conforti,        Cornu\'ejols        and
Vu\'skovi\'c~\cite{conforti.c.v:square}:  every   Berge  graph  either
falls  in  a \emph{basic  class}  or  has  a \emph{structural  fault}.
Before  stating this  more precisely,  let us  say that  a \emph{basic
class} of graphs is a class of graphs that are proved to be perfect by
some  classical coloring  argument.   A \emph{structural  fault} in  a
graph is something  that cannot occur in a  minimal counter-example to
the perfect graph conjecture.  The basic classes used by Chudnovsky et
al.  are  the bipartite graphs,  their complement, the  line-graphs of
bipartite graphs,  their complement, and the  double split-graphs. The
structural  faults used by  Chudnovsky et  al.  are  the \emph{2-join}
(first          defined          by          Cornu\'ejols          and
Cunningham~\cite{cornuejols.cunningham:2join}),  the  \emph{even  skew
partition}     (a    refinement     of        Chv\'atal's     skew
partition~\cite{chvatal:starcutset})  and the  \emph{homogeneous pair}
(first defined  by Chv\'atal and Sbihi~\cite{chvatal.sbihi:bullfree}).
We do not give  here the precise definitions as far as  we do not need
them.

Despite  those  breakthroughs,  some  conjectures about  Berge  graphs
remain  open.   An \emph{even  pair}  in  a graph  $G$  is  a pair  of
non-adjacent vertices such that  every chordless path between them has
even length  (number of edges).  Given  two vertices $x,y$  in a graph
$G$, the  operation of \emph{contracting} them means  removing $x$ and
$y$ and  adding one vertex with  edges to every  vertex of $G\setminus
\{x,y\}$ that is  adjacent in $G$ to at least one  of $x,y$; we denote
by $G/xy$  the graph  that results from  this operation.   Fonlupt and
Uhry proved the following:

\begin{thm}[Fonlupt and Uhry~\cite{fonlupt.uhry:82}]
  \label{th:fonuhr}
  If $G$ is a perfect graph and $\{x,y\}$ is an even pair in $G$, then
  the  graph $G/xy$  is  perfect  and has  the  same chromatic  number
  as~$G$.
\end{thm}

\noindent Meyniel also proved the following:

\begin{thm}[Meyniel,~\cite{meyniel:87}]
  \label{th:meyniel}
  Let $G$ be a minimal imperfect graph. Then $G$ has no even pair.
\end{thm}

So even pairs can be  consider as a ``structural fault'', with respect
to a proof  of perfectness for some classes  of graphs.  This approach
for     proving     perfectness     has     been     formalised     by
Meyniel~\cite{meyniel:87}:  a \emph{strict  quasi-parity}  graph is  a
graph such  that every induced subgraph  either is a clique  or has an
even  pair.  By  Theorem~\ref{th:meyniel},  every strict  quasi-parity
graph is perfect.  Many classical  families of perfect graphs, such as
Meyniel  graphs, weakly  chordal graphs,  perfectly  orderable graphs,
Artemis         graphs,        are         strict        quasi-parity,
see~\cite{everett.f.l.m.p.r:ep,nf:artemis}.    A   \emph{quasi-parity}
graph is a  graph $G$ such that for every induced  subgraph $G'$ on at
least two vertices,  either $G'$ has an even  pair, or $\overline{G'}$
has an  even pair.  By  Theorems~\ref{th:meyniel} and~\ref{th:lovasz},
we  know that  quasi-parity graphs  are perfect.   Quasi-parity graphs
graphs  include  every strict  quasi  parity  graphs,  and also  other
classes          of          graphs:          bull-free          Berge
graphs~\cite{figuereido.m.p:bullfree},       bull-reducible      Berge
graphs~\cite{figuereido.m.v:bullred}.

\begin{figure}
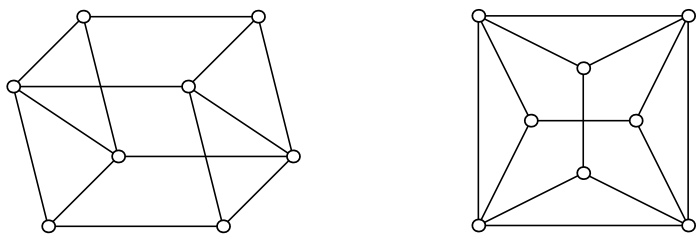

  \includegraphics[height=3cm]{oddpair.1}
  \rule{2cm}{0cm}
  \includegraphics[height=3cm]{oddpair.3}
  \caption{The double-diamond and $L(K_{3,3}\setminus e)$\label{fig:ddlk33e}}
\end{figure}

There are  interesting open  problems about quasi-parity  graphs.  Say
that a graph  is a \emph{prism} if it  consists of two vertex-disjoint
triangles  (cliques  of size  $3$)  with  three vertex-disjoint  paths
between them, and with no other  edges than those in the two triangles
and   in   the   three   paths.   (Prisms   were   called   stretchers
in~\cite{everett.f.l.m.p.r:ep}     and     3PC($\Delta,     \Delta$)'s
in~\cite{conforti.c.v:square}). A prism  is said to be long  if it has
at least  7~vertices.  The double-diamond  and $L(K_{3,3}\setminus e)$
are the graphs depicted figure~\ref{fig:ddlk33e}.  Let us now recall a
definition:                 a                 graph                 is
\emph{bipartisan}~\cite{chudvovsky.r.s.t:progress}   if  in   $G$  and
$\overline{G}$ there is no odd  hole, no long prism, no double-diamond
and  no $L(K_{3,3}\setminus  e)$.   The last  50~pages  of the  strong
perfect  theorem paper~\cite{chudvovsky.r.s.t:spgt}  are devoted  to a
proof  of  perfectness for  bipartisan  graphs.   This  part could  be
replaced by a proof of the following conjecture:

\begin{conjecture}[Maffray, Thomas]
  \label{conj.bipar}
  Every bipartisan graph is a quasi-parity graph.
\end{conjecture}

Why  not conjecture that  \emph{every} Berge  graph is  a quasi-parity
graph~?  Simply because this is false. Some counter-examples (like the
smallest  one: $L(K_{3,3}\setminus  e)$),  were known  since the  very
beginning  of the  study of  even-pairs.  Hougardy  found  an infinite
class of counter-examples:

\begin{thm}[Hougardy, \cite{hougardy:95}]
  Let  $G$ be  the line-graph  of a  3-connected graph.  Then  $G$ and
  $\overline{G}$ have no even pair.
\end{thm}

The aim  of this  paper is to  investigate the following  question: is
there an  ``even-pair-like'' notion that  could be a  structural fault
present in every Berge graph~?   We know that line-graphs of bipartite
graphs are likely  to be without even pairs.   So, they certainly form
one of the first class where  we have to find something.  On the other
hand, a  bipartite graph $B$  with at least  3 vertices always  has an
even  pair:  consider two  vertices  $a,b$ in  the  same  side of  the
bipartition.  What happens to this even pair $\{a,b\}$ in $L(B)$~? All
the edges incident to $a$ form  a clique $K_a$ of $L(B)$, and there is
a similar  clique $K_b$.  Moreover,  every induced path from  $K_a$ to
$K_b$  with no  interior  vertices  in $V(K_a)  \cup  V(K_b)$ has  odd
length.   This leads  us to  the following  definition: let  $K_1$ and
$K_2$ be two cliques of a graph  $G$.  We say that an induced path $P$
is \emph{external  from $K_1$ to $K_2$}  if $P$ has  one end-vertex in
$K_a$, one end-vertex in $K_b$, and all the other possible vertices in
$V(G) \setminus (V(K_1) \cup V(K_2))$.   The pair $\{K_1, K_2\}$ is an
\emph{odd  pair of  cliques} if  every external  induced  path between
$K_1$ and $K_2$ has odd length.  Note that if $\{K_1, K_2\}$ is an odd
pair of  cliques, then $K_1$ and  $K_2$ are disjoint  since a possible
common vertex  would be an  external path of length~0.   Similarly, we
say that $\{K_1,  K_2\}$ is an \emph{even pair  of cliques} when every
external induced  path between  $K_1$ and $K_2$  has even  length.  We
propose the following two conjectures:

\begin{conjecture}
  \label{conj.struct}
  Let $G$ be a Berge graph on at least two vertices. Then either:
  \begin{itemize}
  \item
    $G$ or $\overline{G}$ has an even pair.
  \item
    $G$ or  $\overline{G}$ has an  odd pair of cliques  $\{K_1, K_2\}$
    such that $K_1, K_2$ are maximal cliques of $G$.
  \end{itemize}
\end{conjecture}

\begin{conjecture}
  \label{conj.mini}
  Let $G$ be  a mimilal imperfect graph.  Then $G$ has  no odd pair of
  cliques  $\{K_1, K_2\}$, such  that $K_1,  K_2$ are  maximal cliques
  of~$G$.
\end{conjecture}

Clearly, between two  maximal cliques of an odd  hole, there exists an
external induced path  of even length. Between two  maximal cliques of
an odd  antihole, there exists  an external induced path  of length~2.
But by  the strong perfect  graph theorem, the only  minimal imperfect
are the odd holes and the  odd antiholes. Thus the conjecture above is
true. But we  would like a proof that does not  use the strong perfect
graph theorem.

As    already    mentioned,    it     is    easy    to    see    that
Conjecture~\ref{conj.struct}  holds for bipartite  graphs, line-graphs
of bipartite graphs, and their  complement. Let us prove that it holds
also  for the last  basic class:  double split graphs.  A \emph{double
split graph} (defined  in~\cite{chudvovsky.r.s.t:spgt}) is  any graph
$G$  that  can  be  constructed  as  follows.  Let  $m,n  \geq  2$  be
integers. Let $A = \{a_1, \dots, a_m\}$, $B= \{b_1, \dots, b_m\}$, $C=
\{c_1, \dots, c_n\}$, $D= \{d_1,  \dots, d_n\}$ be four disjoint sets.
Let $G$ have  vertex set $A\cup B \cup  C \cup D$ and edges  in such a
way that:

\begin{itemize}
\item 
  $a_i$ is  adjacent to  $b_i$ for $1  \leq i  \leq m$.  There  are no
  edges between $\{a_i, b_i\}$ and $\{a_{i'}, b_{i'}\}$ for $1\leq i <
  i' \leq m$.
\item 
  $c_j$ is non-adjacent to $d_j$ for  $1 \leq i \leq m$. There are all
  four  edges  between $\{c_j,  d_j\}$  and  $\{c_{j'}, b_{j'}\}$  for
  $1\leq j < j' \leq n$.
\item
  There  are exactly  two  edges between  $\{a_i,  b_i\}$ and  $\{c_j,
  d_j\}$ for  $1\leq i \leq  m$ and  $1 \leq j  \leq n$ and  these two
  edges are disjoint.
\end{itemize}

If $G$ is a double split graph with the notation of the definition, we
may assume  up to  a relabeling  of the $c_j,  d_j$'s that  $a_1$ sees
every $c_j$  and that $b_1$ sees  every $d_j$ (if this  fails for some
$j$, just swap $c_j,  d_j$). Now it is easy to see  that $K_a = \{a_1,
c_1,  \dots, c_n\}$  and $K_b  = \{b_1,  d_1, \dots,  d_n\}$  are both
maximal cliques of  $G$.  The only possible external  induced paths of
length greater  than~1 from  $K_a$ to $K_b$,  are paths from  $c_j$ to
$d_j$ for some $j$.  But such a  path must start in $c_j$, and then go
to some $a_i$,  and then the only  option is to go to  $b_i$, and then
back  to $d_j$.   So,  every external  path  from $K_a$  to $K_b$  has
length~$1$  or~$3$.  So,  $\{K_a, K_b\}$  is  an odd  pair of  maximal
cliques.  Note that there exist  double split graphs that have no even
pair:  $L(K_{3,3}\setminus e)$  is  an example  and arbitrarily  large
examples exist.  However, Conjecture~\ref{conj.struct} holds for every
basic graph.

\section{Odd pairs of cliques in line-graphs of bipartite graphs} 

We  observed in the  introduction that  every line-graph  of bipartite
graph has an odd pair of cliques. In this section, we will see that we
can say something much stronger. But we first need some information on
the structure of line-graphs of bipartite graphs.

\begin{figure}
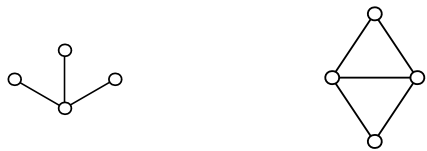

  \parbox[c]{0cm}{\rule{0cm}{2cm}}
  \parbox[c]{4cm}{\center\includegraphics[height=1cm]{oddpair.4}}
  \parbox[c]{4cm}{\center\includegraphics[height=2cm]{oddpair.2}}
  \caption{The claw and the diamond}
\end{figure}

The facts  stated in this paragraph  need careful checking,  but we do
not     prove      them     since     they      are     well     known
(see~\cite{beineke:linegraphs}  and~\cite{harary.holzmann:lgbip}). Let
us consider a graph $G$ that contains no claw and no diamond.  Let $v$
be a vertex  of $G$. Either $v$ belongs to  exactly one maximal clique
of $G$, or $v$ belongs to exactely two maximal cliques of~$G$.  In the
second case, the  intersection of the two cliques  is exactly $\{v\}$.
Let us build a new graph $R$.  Every maximal clique of $G$ is a vertex
of $R$.   Such vertices of  $R$ are called the  \emph{clique vertices}
of~$R$. Every vertex of $G$ that  belongs to a single clique of $G$ is
also of vertex  of $R$. Such vertices of  $R$ are called \emph{pendent
vertices} of $R$.   We add an edge between two  clique vertices of $R$
whenever the two corresponding cliques in $G$ do intersect.  We add an
edge between  a clique vertex $u$ of  $R$ and a pendent  vertex $v$ of
$R$ whenever the pendent vertex $v$ is a vertex of $G$ that belongs to
$u$ seen  as a clique of  $G$. Note that  a pendent vertex of  $R$ has
always degree~1  (the converse  is not true  when $G$ has  a connected
component that  consists in a single  vertex). One can  check that $R$
has no triangle  and that $G$ is isomorphic to  $L(R)$.  This leads us
to the following well known theorem:

\begin{thm}
  \label{pair.l.LB}
  Let $G$  be a  graph.  There exists  a triangle-free graph  $R$ such
  that $G=L(R)$ if and only if $G$ contains no claw and no diamond.
\end{thm}

The following theorem shows that the maximal cliques of the line-graph
of a bipartite graph behave like  the vertices of a bipartite graph in
quite a strong sense.

\begin{thm}
  \label{pair.l.caracennemie}
  Let $G$  be a graph  with no  claw and no  diamond. Then $G$  is the
  line-graph of a  bipartite graph if and only  if the maximal cliques
  of $G$  may be partitioned into two  sets $A$ and $B$  such that for
  every distinct maximal cliques $K_1, K_2$ of $G$ we have:
  \begin{itemize}
  \item
    If $K_1\in A$  and $K_2\in A$, then $\{K_1, K_2\}$  is an odd pair
    of cliques.
  \item
    If $K_1\in B$  and $K_2\in B$, then $\{K_1, K_2\}$  is an odd pair
    of cliques.
  \item
    If $K_1\in A$ and $K_2\in B$,  then $\{K_1, K_2\}$ is an even pair
    of cliques.
  \end{itemize} 
\end{thm}

\begin{proof}
  By  the discussion  above, we  know that  $G$ is  isomorphic  to the
  line-graph of a triangle-free graph  $R$. So, we may assume that $R$
  is  built  from  $G$  like  in the  construction  described  in  the
  discussion above.

  If $R$ is bipartite, then the clique vertices of $R$ are partitioned
  into  two  stable  sets $A$  and  $B$.   This  partition is  also  a
  partition of  the maximal  cliques of~$G$. So,  let $K_1 \in  A$ and
  $K_2  \in A$ be  two maximal  cliques of~$G$.   Note that  $K_1$ and
  $K_2$ are also non-adjacent vertices of~$R$.  So, we know that $K_1$
  and $K_2$ are disjoints cliques  of~$G$.  If there exists an induced
  path of~$G$, of even length,  external from $K_1$ to~$K_2$, then the
  interior vertices of this path (which has length at least~2) are the
  edges of  the interior of a path  of $R$ of odd  length, linking the
  vertex $K_1$ to the  vertex~$K_2$.  This contradicts the bipartition
  of~$R$.  So,  every external induced  path in $G$ between  $K_1$ and
  $K_2$ is of odd length, in  other words, $K_1$ and $K_2$ form an odd
  pair of cliques.   By the same way, we prove that  if $K_1\in B$ and
  $K_2\in  B$, then  $K_1$  and $K_2$  form  an odd  pair of  cliques.
  Similarly, if $K_1 \in A$ and $K_2 \in B$, then $K_1$ and $K_2$ form
  an even pair of cliques.

  If $R$ is not  bipartite, then $R$ has an odd hole  $H$ of length at
  least~5  (because  $R$  is  triangle-free).  Let  $v_1,  v_2,  \dots
  v_{2k+1}$  be the  vertices of  $H$ in  their natural  order.  Every
  vertex  of $H$  has degree  at least~2,  and therefore  is  a clique
  vertex of $R$. So, every vertex $v_i$ is in fact a maximal clique of
  $G$.  If  one manages to partition  the maximal cliques  of $G$ into
  two  sets $A$ and  $B$ as  indicated in  the lemma,  two consecutive
  cliques $v_i$  and $v_{i+1}$ are  not disjoint.  So, they  cannot be
  both  in $A$  or both  in  $B$.  So  in the  sequence $(v_1,  \dots,
  v_{2k+1}, v_1,\dots )$  every second clique is in  $A$ and the other
  ones are  in $B$.  But this  is impossible because  there is  an odd
  number of $v_i$'s.
\end{proof}

\section{An operation that preserves perfectness}

We know  that the contraction of  an even pair $\{x,y\}$  in a perfect
graph  $G$   yields  another  perfect   graph.   What  would   be  the
corresponding operation  in $L(G)$ for  an odd pair of  cliques~?  The
edges  incident  to $x$  form  a clique  $K_x$  of  $L(G)$, and  those
incident to  $y$ form a clique  $K_y$.  The contracted  vertex $xy$ in
$G/xy$ is  incident to the edges that  were incident to $x$  or $y$ in
$G$, and  so becomes  in $L(G)$  a clique obtained  by adding  an edge
between every  vertex of $K_x$ and  every vertex of $K_y$.   So let us
define  the  following  operation  for  any graph  $G$  and  any  pair
$\{K_1,K_2\}$ of  disjoint cliques  of $G$: just  add an  edge between
every  vertex in  $K_1$ and  every vertex  in $K_2$  (if they  are not
adjacent).  The graph obtained  is denoted by $G_{K_1\equiv K_2}$.  We
will see that this operation  preserves perfectness when applied to an
odd pair of cliques.  Before  this, we need a technical lemma, roughly
saying that in $G_{K_1\equiv K_2}$,  there is no other big clique than
the clique induced by~$V(K_1) \cup V(K_2)$:

\begin{lem} \label{pair.l.grosseclique}
  Let $\{K_1, K_2\}$ be an odd pair of cliques in a graph $G$. Let $K$
  be  a  clique of  $G_{K_1  \equiv K_2}$.  There  are  then only  two
  possibilities:
  \begin{itemize}
  \item
    $K$ is a  clique of $G$.
  \item
    $V(K) \subseteq V(K_1) \cup V(K_2)$.
  \end{itemize}
\end{lem}

\begin{proof}
  If $K$ is not  a clique of $G$, then $K$ contains  at least a vertex
  $v_1$ of  $K_1$ and a  vertex $v_2$ of  $K_2$ that are  not adjacent
  in~$G$.   Moreover,  if  $V(K)$  if  not included  in  $V(K_1)  \cup
  V(K_2)$, then  $K$ contains a vertex  $v$ that is  neither in $K_1$,
  nor in $K_2$, and that sees  $v_1$ and $v_2$.  But then, $ v_1 \!-\!
  v \!-\!   v_2 $ is an external  induced path of $G$,  of even length
  from $K_1$ to $K_2$, a contradiction.
\end{proof}

The proof of the next theorem looks like the proof of Fonlupt and Uhry
for  Theorem~\ref{th:fonuhr}.    For  Theorem~\ref{th:fonuhr},  it  is
needed to prove by a  bichromatic exchange that some vertices may have
\emph{the same} color  in some optimal coloring of  a graph.  The only
possible obstruction to  this exchange is a path  of odd length between
them, contradicting the definition of an even pair of vertices. In our
theorem, at  a certain  step we will  need to  prove that there  is an
optimal   coloring  that   gives  \emph{different}   colors   to  some
vertices. The only obstruction to this will be an induced path of even
length, contradicting the definition of odd pairs of cliques.

\begin{thm}
  \label{pair.l.preserve}
  Let $G$ be a perfect graph and  let $\{K_1, K_2\}$ be an odd pair of
  cliques of $G$. Then $G_{K_1 \equiv K_2}$ is a perfect graph.
\end{thm}

\begin{proof}
  Let $H'$ be an induced subgraph of $G_{K_1 \equiv K_2}$.  Let $H$ be
  the induced subgraph of $G$  that has the same vertex-set than $H'$.
  Clearly, $V(K_1) \cap V(H)$ and  $V(K_2) \cap V(H)$ form an odd pair
  of cliques  in $H$  and $H'=H_{(K_1 \cap  H) \equiv (K_2  \cap H)}$.
  So, to prove  the theorem, it suffices to  check $\chi(G_{K_1 \equiv
  K_2}) =  \omega(G_{K_1 \equiv  K_2})$.  Let us  suppose that  $G$ is
  colored with $\omega(G)$ colors.  We  look for a coloring of $G_{K_1
  \equiv K_2}$ with $\omega(G_{K_1 \equiv K_2})$ colors.

  Let us  first color the  vertices that are  neither in $K_1$  nor in
  $K_2$: we  give them their  color in $G$.  If  $\omega(G_{K_1 \equiv
  K_2}) > \omega(G)$, then by Lemma~\ref{pair.l.grosseclique}, we know
  that $V(K_1) \cup V(K_2)$ induces the only maximum clique of $G_{K_1
  \equiv K_2}$.  So, whatever the sizes of $K_1, K_2$, we take $\gamma
  = \max  (0, |V(K_1) \cup V(K_2)|  - \omega(G))$ new  colors.  We use
  them to color $\gamma$ vertices in $V(K_1) \cup V(K_2)$.  So, we are
  left with $|V(K_1)  \cup V(K_2)| - \gamma$ vertices  in $V(K_1) \cup
  V(K_2)$: let  us give them their  color in~$G$.  We  may assume that
  there is  a vertex $v_1$ in $K_1$  and a vertex $v_2$  in $K_2$ with
  the same  color (say  red) for otherwise  we have  an $\omega(G_{K_1
  \equiv K_2})$-coloring of $G_{K_1 \equiv K_2}$ and the conclusion of
  the lemma holds.

  So there is a  color used in $G$ (say blue) that  is used neither in
  $K_1$ nor in $K_2$.  Let $C$ be  the set of vertices of $G$ that are
  red or blue. The set $C$  induces a bipartite subgraph of $G$ and we
  call $C_1$  the connected component  of $v_1$ in this  subgraph.  If
  $v_2 \in C_1$, then a shortest  path in $C_1$ from $v_1$ to $v_2$ is
  an induced  path of $G$, of  even length from $K_1$  to $K_2$.  This
  path is  external because  there is no  blue vertex in  $V(K_1) \cup
  V(K_2)$.   This  contradicts  the  definititon  of an  odd  pair  of
  cliques, so $v_1$ and $v_2$  are not in the same connected component
  $C_1$.  So,  we can exchange the  colors red and blue  in $C_1$, and
  give the color  blue to $v_1$, without changing  the color of $v_2$.
  We can do this again as long as there are vertices of the same color
  in  $K_1 \cup  K_2$.  Finally,  we obtain  an  $\omega(G_{K_1 \equiv
  K_2})$-coloring of $G_{K_1 \equiv K_2}$.
\end{proof}

\section{Odd pairs of cliques in minimal imperfect graphs}

In this  section, we will see  that in a minimal  imperfect graph $G$,
there  is no  pair of  odd cliques  $(K_1, K_2)$  with  $|K_1|+|K_2| =
\omega(G)$.   This will  be proven  without using  the  strong perfect
graph theorem. We first need some results on minimal imperfect graphs.

\begin{thm}[Lov\'asz, \cite{lovasz:pg}]
  \label{graphespar.t.lovasz}
  A graph  $G$ is perfect  if and only  if for every  induced subgraph
  $G'$ we have $\alpha(G') \omega(G') \geq |V(G')|$.
\end{thm}

Lov\'asz also introduced an important notion. Let $p, q \geq 1$ be two
integers. A   graph  $G$   is  \emph{$(p,
q)$-partitionable} if  and only  if for every  vertex $v$ of  $G$, the
graph $G \setminus v$ can be  partitioned into $p$ cliques of size $q$
and  also into  $q$ stable  sets  of size  $p$.  The  theorem to  come
follows from Theorem~\ref{graphespar.t.lovasz}:

\begin{thm}[Lov\'asz, \cite{lovasz:pg}]
  \label{graphespar.t.minimparfait}
  Let  $G$ be a minimal imperfect graph. Then  $G$ is partitionable.
\end{thm}

Partitionable    graphs    have    several   interesting    properties
(see~\cite{preissmann.sebo:minimal}        for        a       survey).
Padberg~\cite{padberg:74} proved the  following in the particular case
of minimal imperfect graphs:

\begin{thm}[Bland, Huang, Trotter \cite{bland.h.t:79}]
  \label{graphespar.t.partitionable}
  Let $G$ be a   graph $(p, q)$-partitionable with $n=p
  q +1$ vertices. Then:
  \begin{enumerate}
  \item
    $\alpha(G) = p$ and $\omega(G) = q$.
  \item 
    \label{graphespar.o.nc}
    $G$ has exactly $n$ cliques of size~$\omega$.
  \item
    $G$ has exactly $n$ stable sets of size~$\alpha$.
  \item 
    Every  vertex  of  $G$  belongs  to exactly  $\omega$  cliques  of
    size~$\omega$.
  \item 
    Every vertex  of $G$  belongs to exactly  $\alpha$ stable  sets of
    size~$\alpha$.
  \item 
    \label{graphespar.o.cs}
    Every clique of $G$ of  size $\omega$ is disjoint from exactly one
    stable set of $G$ of size~$\alpha$.
  \item 
    \label{graphespar.o.sc}
    Every stable set of~$G$ of  size $\alpha$ is disjoint from exactly
    one clique of~$G$ of size~$\omega$.
  \item 
    \label{graphespar.o.unicolor}
    For every  vertex $v$  of $G$,  there is a  unique coloring  of $G
    \setminus v$ with $\omega$ colors.
  \end{enumerate}
\end{thm}

If $K_1$ and  $K_2$ are two disjoint subcliques of  a clique $K$, then
they form an odd pair of cliques.  In this case, we say that $K_1$ and
$K_2$  form  a \emph{trivial  odd  pair  of  cliques}.  The  following
theorem is a particular case of Conjecture~\ref{conj.mini}:

\begin{thm} 
  \label{pair.l.omega}
  Let $G$ be  a minimal imperfect graph.  Let $\{K_1,  K_2\}$ be a non
  trivial  odd  pair  of  cliques  of  $G$.   Then  $|K_1|+|K_2|  \neq
  \omega(G)$.
\end{thm}

\begin{proof}
  Suppose      $|K_1|     +      |K_2|      =     \omega(G)$.       By
  Lemma~\ref{pair.l.grosseclique},   $\omega(G_{K_1  \equiv   K_2})  =
  \omega(G)$.  Moreover, $\alpha(G_{K_1  \equiv K_2}) \leq \alpha(G)$.
  And   by  Theorem~\ref{graphespar.t.lovasz},   we   have  $\alpha(G)
  \omega(G) < |V(G)|$.

  By the definition, every induced  subgraph of $G$ is perfect. So, by
  Theorem~\ref{pair.l.preserve},  every  induced  subgraph of  $G_{K_1
  \equiv K_2}$  is perfect.  Note  that the $\omega$-clique  $K_1 \cup
  K_2$ of  $G_{K_1\equiv K_2}$  is not a  clique of $G$  since $\{K_1,
  K_2\}$ is  not a trivial odd  pair of cliques.  So,  by counting the
  cliques  and by the  fact that  $G$ is  partitionable, we  know that
  $G_{K_1   \equiv   K_2}$    is   not   partitionable   (because   of
  Property~(\ref{graphespar.o.nc})                                   of
  Theorem~\ref{graphespar.t.partitionable}).   All  its subgraphs  are
  perfect, so  by Theorem~\ref{graphespar.t.minimparfait}, we  know it
  is perfect.  But we have: %
  \[
  \alpha(G_{K_1  \equiv K_2}) \omega(G_{K_1  \equiv K_2})  
  \leq \alpha(G) \omega(G)  
  < |V(G)| =  |V(G_{K_1 \equiv K_2})|
  \]

  \noindent This contradicts Theorem~\ref{graphespar.t.lovasz}.
\end{proof}

By the preceding theorem, if $\{K_1,  K_2\}$ is an odd pair of cliques
in a minimal imperfect graph $G$, there are two cases:

\begin{itemize}
\item{$|K_1|+|K_2|<\omega(G)$}

  \noindent In  this case,  interestingly, the edges  that we  add when
  constructing $G_{K_1\equiv  K_2}$ do not  create any $\omega$-clique
  by  Lemma~\ref{pair.l.grosseclique}.  Moreover,  these edges  do not
  destroy any $\alpha$-stable.  Let us prove this:

\begin{proof}
  Suppose that an $\alpha$-stable set of~$G$ is destroyed.  This means
  that there  exists two vertices  $v_1\in K_1$ and $v_2\in  K_2$ that
  are     in    some     $\alpha$-stable    set~$S$     of~$G$.     By
  Property~(\ref{graphespar.o.sc})                                   of
  Theorem~\ref{graphespar.t.partitionable},     there    exists    one
  $\omega$-clique  $K$ disjoint  from~$S$.   Let $v\in  V(K)$. By  the
  definition  of   partitionable  graphs,  $G  \setminus   v$  can  be
  partitioned into  $\omega$ stable sets  of size $\alpha$.   At least
  one  of these  stable sets  (say~$S'$) is  disjoint from  $K$, since
  $K\setminus    v$    contains    $\omega    -    1$~vertices.     By
  Property~(\ref{graphespar.o.cs})                                   of
  Theorem~\ref{graphespar.t.partitionable}, we know that $S'=S$. So we
  have found  in~$G$ a  vertex $v$  such that $G  \setminus v$  can be
  optimaly colored giving to $v_1$  and $v_2$ the same color, say red.
  But since  $|K_1|+|K_2|<\omega(G)$, there exists a  color (say blue)
  that is not used in $K_1  \cup K_2$. By a bichromatic exchange (like
  in  the  proof  of  Theorem~\ref{pair.l.preserve}), we  can  find  a
  coloring of $G  \setminus v$ that gives the same  red color to $v_1$
  and color  blue to  $v_2$ (if  such an exchange  fails, there  is an
  external  induced path  of even  length between  $K_1$ and  $K_2$, a
  contradiction).   Finaly   we  found  two   different  colorings  of
  $G\setminus             v$.              This            contradicts
  Property~(\ref{graphespar.o.unicolor})                             of
  Theorem~\ref{graphespar.t.partitionable}.
\end{proof}

\noindent  So   $G_{K_1  \equiv  K_2}$  is   a  partitionable  graph.
Seemingly, this does not lead to a contradiction.

\vspace{1ex}

\item{$|K_1|+|K_2|>\omega(G)$}\\
  \noindent 
  In  this case,  by  Lemma~\ref{pair.l.grosseclique}, $G_{K_1  \equiv
  K_2}$ has a unique maximum clique: $K_1\cup K_2$.  This graph is not
  partitionable,  all its  induced subgraphs  are perfect,  so  it is
  perfect. One more time, this does not seem to lead to contradiction.
\end{itemize}

\section{Odd pairs of cliques in Berge graphs}

To  prove  Conjecture~\ref{conj.struct},  one  could try  to  use  the
approach    that    worked   for    the    decomposition   of    Berge
graphs~\cite{chudvovsky.r.s.t:spgt}: first, consider the case when $G$
has a ``substantial'' line-graph $H$  as an induced subgraph.  We know
that      $H$     has     an      odd     pair      of     cliques~(by
Theorem~\ref{pair.l.caracennemie}).   Then, one  could hope  that this
pair  of cliques  is likely  to  somehow ``grow''  to an  odd pair  of
cliques of the whole graph.  A  \emph{star-cutset} in a graph $G$ is a
set $C$ of vertices such  that $G\setminus C$ is disconnected and such
that there exists a vertex in  $C$ that sees all the other vertices of
$C$.       Star     cutsets      have      been     introduced      by
Chv\'atal~\cite{chvatal:starcutset},  who  proved   that  they  are  a
``structural fault'' that cannot occur in minimal imperfect graph.  It
is known however that some non-basic Berge graphs have no star-cutset.
The following lemma shows that there is something wrong in the idea of
making the odd pair cliques ``grow'':  it can work only in graphs that
have a star-cutset.

\begin{lem}
  Let $\{K_1, K_2\}$ be an odd pair of cliques of a graph $G$. Suppose
  that $K_2$  is a  maximal clique of  $G$. Let  $K'_1 \neq K_1$  be a
  sub-clique of $K_1$.  If $\{K'_1,  K_2\}$ is an odd pair of cliques,
  then $G$ has a star cutset.
\end{lem}

\begin{proof}
  Let $a  \in K'_1$  and $b  \in K_2$ be  non adjacent  vertices (they
  exist because $K_2$  is maximal).  Let $c$ be  any vertex of $V(K_1)
  \setminus  V(K'_1)$.  We  are going  to show  that $\{a\}  \cup N(a)
  \setminus  \{c\}$ is a  cutset of  $G$ separating  $c$ from  $b$. To
  prove this,  we check that  every induced path  $P$ from $c$  to $b$
  that has  no interior  vertex in $K_1$  contains a neighbour  of $a$
  different of $c$.  Indeed:
  
  If the interior of $P$ contains no vertex of $K_2$, then $P$ has odd
  length  because $\{K_1,  K_2\}$ is  an odd  pair of  cliques.  Since
  $\{K'_1, K_2\}$ is  an odd pair of cliques, there is  a chord in the
  even-length path $(a,  c, \dots, b)$, and this  chord is between $a$
  and a vertex of the interior of $P$.

  If the interior of $P$ contains  a vertex of $K_2$, then this vertex
  is the neighbour of $b$ in $P$: we denote it by $d$. We see that $ c
  \!-\! P \!-\!   d $ has odd length because $\{K_1,  K_2\}$ is an odd
  pair of  cliques.  So the path  $(a, c, \dots, d)$  has even length,
  and there is a chord between $a$ and a vertex of the interior of $P$
  (this chord can be $ad$).
\end{proof}

\newpage

%\begin{thebibliography}{1}
%\bibitem{test} A. B. C. Test, \textit{On a Test.} J. of Testing
%\textbf{88} (2000), 100--120.
%\bibitem{latex} G. Gr\"atzer, \textit{Math into \LaTeX.} 3rd Edition,
%Birkh\"auser, 2000.
%\end{thebibliography}

% ------------------------------------------------------------------------

%\subsection*{Acknowledgment}
%Many thanks to our \TeX-pert for developing this class file.
% ------------------------------------------------------------------------
\end{document}